\documentclass[12pt]{amsart}
\usepackage{geometry} 
\usepackage{amsmath,amsthm,amssymb,enumerate,color,graphicx,fancyhdr}
\usepackage{mathrsfs}
\usepackage{latexsym}
\usepackage[all]{xy}
\usepackage{amscd}
\usepackage{pgf,pgfarrows,pgfnodes,pgfautomata,pgfheaps,pgfshade}
\usepackage{amsmath,amssymb}
\usepackage[latin1]{inputenc}
\usepackage{colortbl}
\usepackage{tikz}
\usetikzlibrary{matrix,calc,arrows}
\usepackage{pinlabel}
\usepackage[english]{babel}
\usepackage{overpic}

\newtheorem{thm}{Theorem}[section]
\newtheorem{prop}[thm]{Proposition}
\newtheorem{coro}[thm]{Corollary}

\theoremstyle{definition}
\newtheorem{defn}[thm]{Definition}
\newtheorem{exm}[thm]{Example}
\newtheorem{rem}[thm]{Remark}
\newtheorem{prob}[thm]{Problem}

\title{Lickorish type construction of manifolds over simple polytopes}
\author{Zhi L\"u, Wei Wang and Li Yu}

\thanks{Supported in part by grants from NSFC (No. 11371093, No. 11371188, No. 11301335, No. 11401233, No. 11431009 and
No. 11661131004) and the PAPD (priority academic program development) of Jiangsu higher education
institutions.}
\address{School of Mathematical Sciences, Fudan University, Shanghai,
200433, P.R. China. } \email{zlu@fudan.edu.cn}
\address{College of Information Technology, Shanghai Ocean University, 999 Hucheng Huan Road, 201306, Shanghai, P.R.  China.}
\email{weiwang@amss.ac.cn}
\address{Department of Mathematics and IMS, Nanjing University, Nanjing, 210093, P.R.China}
\email{yuli@nju.edu.cn}

\date{} 

\begin{document}
\begin{abstract}
This paper is a survey on the Lickorish type construction of some kind of closed manifolds over simple convex polytopes. Inspired by Lickorish's theorem, we propose a method to describe certain families of manifolds over simple convex polytopes with torus action. Under this construction, many classical classification results of these families of manifolds could be interpreted by this construction and some further problems will be discussed.
\end{abstract}
\maketitle

\section{What is Lickorish type construction?}
This paper is a survey on the Lickorish type construction of some kind of closed manifolds over simple convex polytopes. We first explain what is ``Lickorish type construction''.

In algebra, it is natural to describe algebraic systems, such as rings and
 algebras, by generators and relations. In geometry and topology, it is
 often convenient to construct spaces from some very special examples by
certain type of operations. We write this construction in terms
of algebraic system by
$$ AS \, \big\{ \textbf{generators}\ |\ \textbf{some operations}
                \big\},$$
                where
	$AS$ is the abbreviation for ``algebraic system''.
One typical example is the following theorem obtained by Lickorish in 1962.

\begin{thm}[Lickorish~\cite{Lickorish}]\label{LickorishThm}
Any orientable closed connected $3$-manifold can be obtained from $S^3$ by a finite number of Dehn surgeries on knots.
\end{thm}
This theorem provides a global viewpoint of the construction of orientable closed connected $3$-manifolds under algebraic system with generators and operations.
We call this kind of construction or description \emph{Lickorish type construction}.
Under this point of view, we can rewrite the above theorem as:
 $$\big\{ \text{{\small All orientable closed connected $3$-manifolds}} \big\} =
 AS \big\{  S^3 \, |\, \text{Dehn surgeries on knots} \big\}.$$

 \vskip .4cm

It is natural to ask: {\em is there any other examples of this construction?}  This is the main motivation of this survey paper.
Under the viewpoint of Lickorish's construction, we survey some
known results of
constructing some families of closed manifolds arising in toric topology.
We hope that such constructions would be further studied with more applications. The reader is referred to~\cite{DJ} and~\cite{BP1} for the backgrounds of toric topology.

\subsection*{Acknowledgements}
The authors are very grateful to the referees for carefully reading this manuscript and providing a number of valuable and helpful suggestions, which led to this version.

\section{Closed manifolds over simple polytopes}
Let $P^n$ be an $n$-dimensional
simple convex polytope in an Euclidean space. Suppose the number of \emph{facets} (codimension-one faces) of $P^n$ is $m$.
According to~\cite{DJ}, one can construct a $T^m$-manifold $\mathcal{Z}_{P^n}$
and a $\mathbb{Z}_2^m$-manifold $\mathbb{R}\mathcal{Z}_{P^n}$ whose
orbit spaces are both $P^n$. Indeed,
let $\{F_1,\cdots, F_m\}$ be the set of facets of $P^n$.
Let $\{e_1,\cdots, e_m\}$ be an unimodular basis of
  $\mathbb{Z}^m$. Define a function
  $\lambda_0 :  \{ F_1,\cdots, F_m \}\rightarrow \mathbb{Z}^m$ by
    \begin{equation}\label{equ:Lambda_0}
     \lambda_0(F_i) = e_i, \ 1\leq i \leq m.
    \end{equation}
      For a proper face $f$ of $P^n$, let $G_f$ denote the subtorus of
      the $m$-dimensional \emph{real torus} $T^m$ determined by the set
       $\{ \lambda_0(F_i) \, |\, f\subset F_i \} \subset \mathbb{Z}^m$ under the exponential map
       $\mathbb{Z}^m \subset \mathbb{R}^m \rightarrow T^m$.
  For any point $p\in P^n$, let $f(p)$ denote the unique face of $P^n$
  that contains $p$ in
  its relative interior. Then by~\cite[Construction 4.1]{DJ},
 the \emph{moment-angle manifold} $\mathcal{Z}_{P^n}$ of $P^n$ is
  defined to be the following quotient space
    \begin{equation} \label{Equ:Real-Moment-Angle}
     \mathcal{Z}_{P^n} := P^n\times T^m \slash \sim
    \end{equation}
  where $(p,g) \sim (p',g')$ if and only if $p=p'$ and $g^{-1}g' \in G_{f(p)}$.  In addition, let
  $$\Theta_{P^n} :
  P^n\times T^m \rightarrow  \mathcal{Z}_{P^n}$$
  be the quotient map. There is a \emph{canonical
  $T^m$-action} on $\mathcal{Z}_{P^n}$ by
   \begin{equation}\label{Equ:Action}
    g' \cdot \Theta_{P^n}(p,g) = \Theta_{P^n}(p,g'g),\ p\in P^n, g,g'\in T^m.
   \end{equation}
  If we replace $T^m$ by the \emph{$\mathbb{Z}_2$-torus} $(\mathbb{Z}_2)^m$ and replace
   $\lambda_0$ by a function that maps $\{ F_1,\cdots, F_m \}$ to a basis of
   $(\mathbb{Z}_2)^m$, we can similarly define
    $\mathbb{R}\mathcal{Z}_{P^n} = P^n\times (\mathbb{Z}_2)^m\slash \sim$
   as~\eqref{Equ:Real-Moment-Angle} and a
    \emph{canonical $(\mathbb{Z}_2)^m$-action} on
   $\mathbb{R}\mathcal{Z}_{P^n}$ as~\eqref{Equ:Action}. \vskip .1cm

V.~Buchstaber (cf.~\cite{BP1}) defines $s_{\mathbb{C}}=s_{\mathbb{C}}(P^n)$ (or $s_{\mathbb{R}}=s_{\mathbb{R}}(P^n)$) to be the maximal dimension of the
subtorus of $T^m$ (or sub-$\mathbb{Z}_2$-torus of $(\mathbb{Z}_2)^m$) that can act freely on $\mathcal{Z}_{P}$ (or $\mathbb{R}\mathcal{Z}_{P}$) through the canonical action. It is easy to see that
$$ s_{\mathbb{C}}\leq m-n,\ \ s_{\mathbb{R}}\leq m-n. $$ \vskip .1cm

Note that $\mathcal{Z}_{P^n}$ and $\mathbb{R}\mathcal{Z}_{P^n}$ are the ``highest level'' manifolds over $P^n$. If we have a subtorus $H_{\mathbb{C}}$ of $T^m$ (or a sub-$\mathbb{Z}_2$-torus $H_{\mathbb{R}}\subset\mathbb{Z}_2^m$) that acts freely on $\mathcal{Z}_{P^n}$ (or $\mathbb{R}\mathcal{Z}_{P^n}$) where rank($H_{\mathbb{C}})\leqslant s_{\mathbb{C}}$ (rank($H_{\mathbb{R}})\leqslant s_{\mathbb{R}}$), we can obtain a $T^m/H_{\mathbb{C}}$-manifold $\mathcal{Z}_{P^n}/ H_{\mathbb{C}}$ (or
$\mathbb{Z}_2^m/H_{\mathbb{R}}$-manifold $\mathbb{R}\mathcal{Z}_{P^n}/ H_{\mathbb{R}}$) with orbit space $P^n$, called \emph{a partial quotient} of $\mathcal{Z}_{P^n}$ (or $\mathbb{R}\mathcal{Z}_{P^n}$). Therefore, one can construct a series of manifolds with real torus (or $\mathbb{Z}_2$-torus) actions whose orbit spaces are all $P^n$ as shown in the following picture.
\[
\setlength{\unitlength}{1381sp}%
\begingroup\makeatletter\ifx\SetFigFont\undefined%
\gdef\SetFigFont#1#2#3#4#5{%
  \reset@font\fontsize{#1}{#2pt}%
  \fontfamily{#3}\fontseries{#4}\fontshape{#5}%
  \selectfont}%
\fi\endgroup%
\begin{picture}(10749,8215)(-299,-7583)
\put(-299,-1261){\makebox(0,0)[lb]{\smash{{\SetFigFont{6}{7.2}{\rmdefault}{\mddefault}{\updefault}free}}}}
\thicklines
{\color[rgb]{0,0,0}\put(7090,-3961){\framebox(3000,4200){}}
}%
\put(3601,-7561){\framebox(4200,1800){}}
\put(3001,-4111){\vector( 4,-3){1872}}
\put(8401,-4111){\vector(-4,-3){1872}}
\put(2101,-136){\makebox(0,0)[lb]{\smash{{\SetFigFont{6}{7.2}{\rmdefault}{\mddefault}{\updefault}$\mathcal{Z}_P$}}}}
\put(2176,464){\makebox(0,0)[lb]{\smash{{\SetFigFont{5}{6.0}{\rmdefault}{\mddefault}{\updefault}${\Bbb C}$--case}}}}
\put(7651,464){\makebox(0,0)[lb]{\smash{{\SetFigFont{5}{6.0}{\rmdefault}{\mddefault}{\updefault}${\Bbb R}$--case}}}}
\put(7351,-211){\makebox(0,0)[lb]{\smash{{\SetFigFont{6}{7.2}{\rmdefault}{\mddefault}{\updefault}${\Bbb R}\mathcal{Z}_P$}}}}
\put(1876,-1936){\makebox(0,0)[lb]{\smash{{\SetFigFont{6}{7.2}{\rmdefault}{\mddefault}{\updefault}$\mathcal{Z}_P/H_{\Bbb C}$}}}}
\put(3616,-6736){\makebox(0,0)[lb]{\smash{{\SetFigFont{6}{7.2}{\rmdefault}{\mddefault}{\updefault}Simple convex polytopes $P$}}}}
\put(7426,-1036){\makebox(0,0)[lb]{\smash{{\SetFigFont{5}{6.0}{\rmdefault}{\mddefault}{\updefault}$\vdots$}}}}
\put(7201,-1936){\makebox(0,0)[lb]{\smash{{\SetFigFont{6}{7.2}{\rmdefault}{\mddefault}{\updefault}${\Bbb R}\mathcal{Z}_P/H_{\Bbb R}$}}}}
\put(7276,-2911){\makebox(0,0)[lb]{\smash{{\SetFigFont{5}{6.0}{\rmdefault}{\mddefault}{\updefault}$1\leq rk H_{\Bbb R}\leq s_{\Bbb R}(P)$}}}}
\put(2251,-1036){\makebox(0,0)[lb]{\smash{{\SetFigFont{5}{6.0}{\rmdefault}{\mddefault}{\updefault}$\vdots$}}}}
\put(1801,-2986){\makebox(0,0)[lb]{\smash{{\SetFigFont{5}{6.0}{\rmdefault}{\mddefault}{\updefault}$1\leq rk H_{\Bbb C}\leq s_{\Bbb C}(P)$}}}}
\put(10151,-1261){\makebox(0,0)[lb]{\smash{{\SetFigFont{6}{7.2}{\rmdefault}{\mddefault}{\updefault}free}}}}
\put(10100,-886){\makebox(0,0)[lb]{\smash{{\SetFigFont{5}{6.0}{\rmdefault}{\mddefault}{\updefault}$H_{\Bbb R}\curvearrowright{\Bbb R}\mathcal{Z}_P$}}}}
\put(7450,-5086){\makebox(0,0)[lb]{\smash{{\SetFigFont{5}{6.0}{\rmdefault}{\mddefault}{\updefault}$s_{\Bbb R}(P)$: Buchstaber invariant}}}}
\put(-180,-5086){\makebox(0,0)[lb]{\smash{{\SetFigFont{5}{6.0}{\rmdefault}{\mddefault}{\updefault}$s_{\Bbb C}(P)$: Buchstaber invariant}}}}
\put(-299,-961){\makebox(0,0)[lb]{\smash{{\SetFigFont{5}{6.0}{\rmdefault}{\mddefault}{\updefault}$H_{\Bbb C}\curvearrowright\mathcal{Z}_P$}}}}
{\color[rgb]{0,0,0}\put(1701,-3961){\framebox(3000,4200){}}
}%
\end{picture}%

\]
\vskip .3cm
 Note that any partial quotient of $\mathcal{Z}_{P^n}$ (or $\mathbb{R}\mathcal{Z}_{P^n}$) can be described by a nondegenerate $\mathbb{Z}^r$-coloring (or $(\mathbb{Z}_2)^r$-coloring) on $P^n$ with $n\leq r\leq m$ in the similar fashion as $\mathcal{Z}_{P^n}$ (or $\mathbb{R}\mathcal{Z}_{P^n}$). A
 \emph{nondegenerate $\mathbb{Z}^r$-coloring} (or \emph{$(\mathbb{Z}_2)^r$-coloring})
  on $P^n$ is a function
  $\mu: \{F_1,\cdots, F_m \} \rightarrow \mathbb{Z}^r$ (or $(\mathbb{Z}_2)^r$) such that
  at any vertex $v = F_{i_1}\cap \cdots \cap F_{i_n}$ of $P^n$, the
  set $\{ \mu(F_{i_1}),\cdots, \mu(F_{i_n}) \}$
   is part of a unimodular basis of $\mathbb{Z}^r$ (or
   $(\mathbb{Z}_2)^r$). Given any nondegenerate
   $\mathbb{Z}^r$-coloring $\mu$ on $P^n$, we can obtain
   an $(n+r)$-dimensional manifold $M(P^n,\mu)$ defined by:
   $$ M(P^n,\mu) = P^n\times T^r\slash\sim  $$
   where $(p,g)\sim (p',g')$ if and only if $p=p'$ and $g^{-1}g'$
   is in the subtorus of $T^r$ determined by the set
 $\{ \mu(F_i) \, |\, f\subset F_i \} \subset \mathbb{Z}^r$.
 There is a canonical action of $T^r$ on $M(P^n,\mu)$ whose orbit space
 is $P^n$. Similarly, given a nondegenerate
  $(\mathbb{Z}_2)^r$-coloring $\mu'$ on $P^n$, we can construct
  an $n$-dimensional manifold $M(P^n,\mu')$ with a canonical
   $(\mathbb{Z}_2)^r$-action whose orbit space is $P^n$. It is not hard to see that any partial quotient of $\mathcal{Z}_{P^n}$ (or $\mathbb{R}\mathcal{Z}_{P^n}$) can be realized as
   $M(P^n,\mu)$ by some nondegenerate $\mathbb{Z}^r$-coloring (or
    $(\mathbb{Z}_2)^r$-coloring) $\mu$ on $P^n$. \vskip .1cm

Inspired by Lickorish's theorem, we propose the following problem.
\begin{prob}
Give a Lickorish type construction for all the partial quotients of
 (real) moment-angle manifolds over simple polytopes described above.
In other words, find some generators and operations that can produce
all such kind of closed manifolds.
\end{prob}
	
In particular, if $T^{m-n}$ (or $(\mathbb{Z}_2)^{m-n}$) can act freely on
	$\mathcal{Z}_{P^n}$ (or $\mathbb{R}\mathcal{Z}_{P^n}$) through the canonical action, the quotient space $\mathcal{Z}_{P^n}\slash T^{m-n}$ (or
	$\mathbb{R}\mathcal{Z}_{P^n}\slash (\mathbb{Z}_2)^{m-n}$) is called a
	\emph{quasitoric manifold} (or a \emph{small cover}) over $P^n$.
	The nondegenerate $\mathbb{Z}^n$-coloring (or
	$(\mathbb{Z}_2)^n$-coloring) on the facets of $P^n$ corresponding to a quasitoric manifold (or a small cover) is also called its
	\emph{characteristic function}.
	These manifolds are introduced by Davis-Januszkiewicz~\cite{DJ}
	as analogues of smooth projective toric variety in the category of closed manifolds with real torus and $\mathbb{Z}_2$-torus actions. In this case, it is also interesting to study the Lickorish type construction of quasitoric manifolds and small covers.

Roughly speaking, these manifolds over simple convex polytopes can be determined by some informations of their polytopes. If these polytopes admits some Lickorish type constructions, then these corresponding manifolds will also admit the induced Lickorish type construction. Therefore, we will discuss the Lickorish type construction of simple convex polytopes in the next section first.\\

\section{Operations on polytopes}
\subsection{Simple polytopes and flips}
One has a suitable ``algebraic system'' to describe all simple polytopes
in any dimension as follows. First of all, let us
 recall the definition of flips on a simple polytope.

     Let $P$ be an $n$-dimensional simple polytope with $m$ facets.
Assume there exists a face $f$ that is a simplex of dimension $a-1$ and let $b=n+1-a$.
One can define the \emph{flip} on $P$ at $f$ as follows (called a
 \emph{flip of type $(a, b)$}).
\begin{itemize}
  \item Let $\Delta^{a-1}=[v_1,\cdots,v_a]$ be a simplex of dimension $a-1$ in
  an $n$-simplex $\Delta^n=[v_1,\cdots, v_{n+1}]$ and
         $\Delta^{b-1}=[v_{a+1},\cdots,v_{n+1}]$ be opposite face of $\Delta^{a-1}$
         in $\Delta^n$.
        The set of points $\{ \frac{1}{2}v_i + \frac{1}{2}v_j \,|\, 1\leq i \leq a,
            a+1\leq j \leq n+1 \}$ spans a hyperplane $L$ which intersect $\Delta^n$
            transversely. Let $H_{a,b}=L\cap \Delta^n$. It is easy to see that
            $H_{a,b}$ is a simple $(n-1)$-polytope
             combinatorially equivalent to the product $\Delta^{a-1}\times \Delta^{b-1}$.
             We have $\Delta^n = (\Delta^{a-1} \diamond H_{a,b}) \cup (\Delta^{b-1}
             \diamond H_{a,b})$ where
             $A \diamond B$ is the convex hull of two sets $A$ and $B$.

	\item Choose a hyperplane to cut off a small neighborhood $N(f)$ of $f$ in $P$
	which is combinatorially equivalent to $\Delta^{a-1} \diamond H_{a,b}$.

	\item Define the flip on $P$ at $f$ be
	\begin{align*}
	\mathrm{flip}_{f}(P)=& (P-N(f))\cup (\Delta^n- \Delta^{a-1} \diamond H_{a,b} )\\
	=&(P-N(f))\cup (\Delta^{b-1} \diamond H_{a,b})
	\end{align*}	
\end{itemize}

   It is shown in~\cite[Corollary 2.7]{BM} that the combinatorial type of
   $\mathrm{flip}_{f}(P)$ is uniquely determined by
  $P$ and $f$.  We also use
   $\digamma_{(a, b)}$ to refer to a general flip of type $(a,b)$ on a
   simple polytope.

 \begin{exm}
     Doing a flip on a simple $n$-polytope $P$ at a vertex $v$ (i.e. a flip of type
     $(1,n)$) is equivalent to ``cutting off'' $v$ from $P$.
     In addition, the flip of the simple $3$-polytope $P$ in Figure~\ref{p:flip-prism}
     at the $1$-simplex $f$ gives us a polytope combinatorially equivalent to the $3$-cube.
   \end{exm}

   \begin{figure}
         \includegraphics[width=0.73\textwidth]{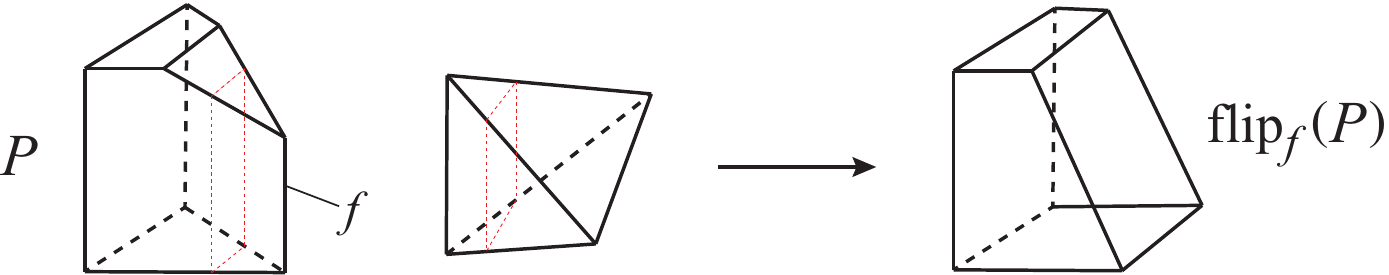}\\
          \caption{Flip at a $1$-simplex}\label{p:flip-prism}
      \end{figure}

 It is shown in~\cite[Lemma 2.3]{BM} that \textsl{up to combinatorial equivalence}
   any simple $n$-polytope can be obtained
   from the $n$-simplex by a finite number of flips.
  We can restate their theorem as follows:
\begin{thm}[Bosio--Meersseman~\cite{BM}] \label{Thm:BM-flips}
\ \vskip .1cm
$\big\{$All simple $n$-polytopes$\big\}/\sim_C \subseteq AS\big\{\Delta^n\big|\,\text{\rm flips}\ \digamma_{(a, b)}, a+b=n+1, 1\leq a,b\leq n\big\}$,
where the equivalent relation $\sim_C$ is up to combinatorial equivalence.
\end{thm}

We warn that the class of simple polytopes is not closed under
   ``combinatorial flips''. Indeed, the result of cutting off a neighborhood of a
simplicial face $f$ of a simple polytope $P$ and gluing the neighborhood of another simplex in
its place may not be a convex polytope (see~\cite[Example 2.10]{BM}).

\subsection{Polytopal spheres and bistellar moves}
Recall that a \emph{polytopal sphere} is defined to be the boundary of a simplicial polytope. A \emph{PL sphere} is a simplicial sphere $K$
which is PL homeomorphic to the boundary of a simplex (i.e.
there is a subdivision of $K$ isomorphic to a subdivision of the boundary of a simplex). Polytopal spheres are all
PL spheres. But there exist non-polytopal PL spheres in dimension $\geq 3$ (e.g. Br\"uckner sphere).
The following operations called ``bisetellar moves'' can
help us to obtain new PL spheres from any given one.
\begin{defn}
Let $K$ be a simplicial $q$-manifold (or any pure $q$-dimensional simplicial complex) and $\sigma \in K$ be
a $(q-i)$-simplex $(0 \leqslant   i \leqslant   q) $ such that $link_K\sigma$  is the boundary $\partial \tau$ of an $i$-simplex $\tau$ that is not a face of $K$. Then the operation $\chi_{\sigma}$ on $K$ defined by
$$\chi_{\sigma}(K):=K \setminus(\sigma* \partial \tau) \cup (\partial \sigma *\tau)$$
is called a bistellar $i$-move.
\end{defn}

  A $q$-dimensional \emph{PL manifold}
   is a simplicial complex $K$ such that
  the link of every nonempty simplex $\sigma$ in $K$ is a PL sphere of dimension
  $q-\dim(\sigma)-1$. Two PL manifolds $K$ and $K'$ are called  \emph{bistellarly equivalent} if we can obtain $K'$
   from $K$ by a finite
  sequence of bistellar moves.
It is easy to see that two bistellarly equivalent PL manifolds
	are PL homeomorphic. The following remarkable result shows that
	the converse is also true.\vskip .1cm
	
	\begin{thm}[Pachner~\cite{Pa87,Pa91}]
	Two PL manifolds are bistellarly equivalent if and only if
	they are PL homeomorphic.
	\end{thm}

Since a simple polytope $P$ and its dual simplicial polytope $P^*$ determine each other, we have
the following one-to-one correspondence:
	$$\{\text{simple polytopes $P$}\}\longleftrightarrow \{\text{polytopal spheres $\partial P^*$}\}.$$
	 It is easy to see that for any $1\leq a \leq n$,
     a $(a,b)$-type flip on $P$ corresponds to
    a bistellar $(a-1)$-move on $\partial P^*$.
    Indeed, any proper face $f$ of $P$ determines a unique
	simplex of $\partial P^*$ denoted by $\sigma_f$, where $\dim(\sigma_f)=n-\dim(f)-1$.
	Then for any simplicial face $f$ of $P$,
	$\chi_{\sigma_f}(\partial P^*)$ corresponds to $\mathrm{flip}_f(P)$.
    So we have the correspondence:
	$$\{\text{filps on simple $n$-polytopes}\}\longleftrightarrow
	\{\text{bistellar moves on polytopal $(n-1)$-spheres}\}.$$
	
	By Pachner's theorem, all the simplicial complexes obtained from
 bistellar moves on $\partial \Delta^n$ are exactly PL $(n-1)$-spheres.
 So we have the following algebraic system of Lickorish type construction of
 all polytopal simplicial spheres.
$$\{ \text{All polytopal $(n-1)$-spheres} \}\subseteq \text{PL  $(n-1)$-spheres} = AS\{\partial\Delta^n\big|\,\text{bistellar moves} \}$$\vskip .1cm

 Moreover, the following theorem
    implies that we do not need
   flips of type $(n,1)$ to obtain a simple $n$-polytope
   $P$ from the $n$-simplex $\Delta^n$.

 \begin{thm}[see~\cite{Ewald78}] \label{Thm:Ewald}
     Let $P$ be simple polytope of dimension $n\geq 3$.
     Then there is a sequence of simple polytopes
     $P_1,\cdots, P_m$ such that $P_1=\Delta^n$, $P_m=P$
     and for $i=1,\cdots,m-1$,
      $\partial P^*_{i+1}$ is obtained from
     $\partial P^*_i$ by a bistellar $k$-move with
     $0\leq k \leq n-2$.
   \end{thm}

So we have the following corollary, up to combinatorial equivalence

  \begin{coro} \label{Cor:less-Flips}
   $\big\{$All simple $n$-polytopes$\big\}/\sim_C\subseteq AS\big\{\Delta^n\big|\,\text{\rm flips}\ \digamma_{(a, b)}, a+b=n+1, 1\leq a \leq n-1, 2\leq b \leq n \big\}$.
  \end{coro}

From the point of view of surgery, flips and bistellar moves are some sort of combinatorial surgeries.  Analogy to Dehn surgeries in
Lickorish's Theorem \ref{LickorishThm}, these two kinds of operations are concrete and constructive. From finite concrete generators, one can use these two kind of operations to construct all objects in the above sets of combinatorial classes.
\begin{rem}
It is known that for a compact PL manifold $M$, the differential structure on $M$ is determined by the homotopy set $[M,PL/O]$. Since $\pi_n(PL/O)=0,n<7$,  it follows that if $M^n$ is a PL $n$-sphere with $n<7$, there is a one-to-one correspondence between PL structures on $M^n$ and smooth structures on $M^n$. The most interesting case is in dimension 4, in this case the classification of PL structure on $S^4$ is equivalent to the classification of smooth structure on $S^4$.

On the other hand, bistellar move doesn't change PL structure. So it may be interesting to find some kind of ``bistellar move" invariants on $S^4$.

\end{rem}
\section{Lifting operations on $P$ to
 $\mathcal{Z}_P$ and $\mathbb{R}\mathcal{Z}_P$}
In section 2, we introduce two families of manifolds over a simple polytope $P$ equipped with some special actions of real torus and $\mathbb{Z}_2$-torus. It is natural to consider the question of lifting the operations on $P$ to these manifolds.
First, let's consider the surgery on ``the highest level''
 $\mathcal{Z}_P$ and $\mathbb{R}\mathcal{Z}_P$.

Suppose $P$ is a simple $n$-polytope with $m$ facets.
Bosio-Meersseman~\cite{BM} describes the equivariant surgery
$\widetilde{\digamma}^{\mathbb{C}}_{(a, b)}$ on $\mathcal{Z}_P$ corresponding to a $(a,b)$-type flip $\digamma_{(a, b)}$ on $P$ at a simplicial face $f$ where $\dim(f)=a-1$.
Let $\pi_{\mathbb{C}}: \mathcal{Z}_P \rightarrow P$ be the orbit map of the canonical $T^m$-action on $\mathcal{Z}_P$. For a small
neighborhood $N(f)$ of $f$ in $P$, it is easy to see that
$\pi^{-1}_{\mathbb{C}}(N(f))\cong S^{2a-1}\times D^{2b}\times T^{m-n-1}$, where $T^{m-n-1}$ are determined by
those facets of $P$ which have no intersection with $f$.
So when removing $N(f)$ from $P$ and glue back $\Delta^{b-1} \diamond H_{a,b}\subset \Delta^n$ in the flip, the corresponding equivariant surgery on $\mathcal{Z}_P$ is given by:
$$\widetilde{\digamma}^{\mathbb{C}}_{(a, b)}(\mathcal{Z}_P)=
\begin{cases}
 \big(\mathcal{Z}_P\setminus (S^{2a-1}\times D^{2b}\times T^{m-n-1})\big)& \\
\cup \big(D^{2a}\times S^{2b-1}\times T^{m-n-1}\big), & \text{ if } a\not=1,n;\\
 \big((\mathcal{Z}_P\times S^1)\setminus(S^1\times D^{2n}\times T^{m-n})\big)&\\
\cup \big(D^2\times S^{2n-1}\times T^{m-n}\big),& \text{ if } a=1;\\
 \Big( \big(\mathcal{Z}_P \setminus (D^2\times S^{2n-1}\times T^{m-n})\big) \cup &\\
 (S^1\times D^{2n}\times T^{m-n}) \Big) \slash S^1,              & \text{ if } a=n.
\end{cases}$$
Note that the $a=n$ case is the converse operation of the $a=1$ case, and
the quotient $\slash S^1$ corresponds to the fact that a simplicial facet is shrunk to a vertex.
 These operations on $\mathcal{Z}_P$ are also given in Buchstaber-Panov~\cite[\S 6.4]{BP1}. \vskip .1cm
 Similarly, let $\pi_{\mathbb{R}}: \mathbb{R}\mathcal{Z}_P \rightarrow P$ be the orbit map of the canonical $(\mathbb{Z}_2)^m$-action on $\mathbb{R}\mathcal{Z}_P$. For a small
neighborhood $N(f)$ of $f$ in $P$, $\pi^{-1}_{\mathbb{R}}(N(f))\cong S^{a-1}\times  D^{b} \times (S^0)^{m-n-1}$.
the equivariant surgery
$\widetilde{\digamma}^{\mathbb{R}}_{(a, b)}$ on $\mathbb{R}\mathcal{Z}_P$ corresponding to the flip of $P$ at $f$ is given by
\[  \widetilde{\digamma}^{\mathbb{R}}_{(a, b)} (\mathbb{R}\mathcal{Z}_{P}) =
 \begin{cases}
\big(\mathbb{R}\mathcal{Z}_{P} \backslash
	     (S^{a-1}\times  D^{b} \times (\mathbb{Z}_2)^{m-n-1}) \big)  & \\
	     \cup
	       \big(D^a\times  S^{b-1} \times (\mathbb{Z}_2)^{m-n-1}\big), & \text{ if } a\not=1, n;\\
 \big((\mathbb{R}\mathcal{Z}_P \times \mathbb{Z}_2) \backslash
  (S^0\times D^n\times (\mathbb{Z}_2)^{m-n}) \big) &\\
  \cup \big(D^1\times S^{n-1}\times (\mathbb{Z}_2)^{m-n}\big),  & \text{ if } a=1;\\
  \Big( \big(\mathbb{R}\mathcal{Z}_P \setminus (D^1\times S^{n-1}\times (\mathbb{Z}_2)^{m-n})\big) \cup &\\
 (S^0\times D^{n}\times (\mathbb{Z}_2)^{m-n}) \Big) \slash \mathbb{Z}_2,              & \text{ if } a=n.
\end{cases}
	   \]\vskip .1cm

Since flips on an $n$-simplex $\Delta^n$ can produce all simple ploytopes (see Theorem~\ref{Thm:BM-flips} and Corollary~\ref{Cor:less-Flips}),
we can use the above equivariant surgeries to
produce all (real) moment-angle manifolds from $S^{2n+1} =\mathcal{Z}_{\Delta^n}$ (or $S^n = \mathbb{R}\mathcal{Z}_{\Delta^n}$).
Moreover, for any PL sphere $K$ we can define
 moment-angle complex $\mathcal{Z}_K$ and real moment-angle complex
 $\mathbb{R}\mathcal{Z}_K$, which generalizes the constructions for
 simple polytopes. Indeed, $\mathcal{Z}_K$ and
 $\mathbb{R}\mathcal{Z}_K$ are still topological manifolds for any PL sphere $K$.

 Furthermore, two moment-angle manifolds $\mathcal{Z}_{P_1}$ and $\mathcal{Z}_{P_2}$ are equivariantly homeomorphic if and only if $P_1$ and $P_2$ are combinatorially equivalent. Moreover, in \cite{BM} equivariant homeomorphism can be strengthened to equivariantly diffeomorphism.

Combing the results of Theorem \ref{Thm:BM-flips} and Buchstaber-Panov~\cite[\S 6.4]{BP1}, we have the following.

\begin{thm} \label{Thm:Surgery}
Lickorish type constructions for moment-angle manifolds and real moment-angle manifolds.
\begin{itemize}
		\item[$(1)$] ${\Bbb C}$-case:
\begin{align*}
	&\ \ \big\{\text{All moment-angle manifolds over simple $n$-polytopes with $m$ facets} \big\} /\sim_H\\
		&\varsubsetneqq
		\big\{\text{All moment-angle complexes over \  PL
		$(n-1)$-spheres with $m$ vertices} \big\} /\sim_H\\
		&\ \ \  = AS\big\{ \mathcal{Z}_{\Delta^n} = S^{2n+1} \,\big|\, \widetilde{\digamma}^{\mathbb{C}}_{(a, b)}, a+b=n+1,
		1\leq a,b \leq n \big\};
\end{align*}
\item[$(2)$] ${\Bbb R}$-case:
\begin{align*}
	&\ \ \big\{\text{All real moment-angle manifolds over simple $n$-polytopes with $m$ facets} \big\}/ \sim_H\\
	&\varsubsetneqq \big\{\text{All real moment-angle complexes over PL $(n-1)$-spheres with }\\
	&\ \ 	\text{$m$ vertices} \big\} /\sim_H
	= AS\big\{\mathbb{R}\mathcal{Z}_{\Delta^n} = S^{n} \,\big| \,
		\widetilde{\digamma}^{\mathbb{R}}_{(a, b)}, a+b=n+1, 1\leq a,b \leq n\big\},
\end{align*}	
	\end{itemize}
\noindent where $\sim_H$ denotes the equivalence relation of equivariant homeomorphism.
\end{thm}

Since $\widetilde{\digamma}_{(a,b)}^{\mathbb{C}}$ and $\widetilde{\digamma}_{(a,b)}^{\mathbb{R}}$ both preserve the equivariant cobordism classes of the corresponding manifolds, we can deduce the following from Corollary~\ref{Cor:less-Flips}.
\begin{coro}
The moment-angle manifold and real moment-angle manifold of any simple polytope are equivariantly cobordant to zero in the category of
compact manifolds
with effective real torus or $\mathbb{Z}_2$-torus actions.
\end{coro}
\begin{proof}
 Let $P$ be a simple $n$-polytope with $m$ facets.
  By Corollary~\ref{Cor:less-Flips} and the definition of
  $\widetilde{\digamma}^{\mathbb{C}}_{(a, b)}$ and $\widetilde{\digamma}^{\mathbb{R}}_{(a, b)}$, the manifolds $\mathcal{Z}_{P}$ and
     $\mathbb{R}\mathcal{Z}_{P}$ are
      equivariantly cobordant to $S^{2n+1}\times T^{m-n-1}$ and
     $S^{n}\times (\mathbb{Z}_2)^{m-n-1}$, respectively.
     There is natural extension of the
      canonical action of $T^m =T^{n+1}\times T^{m-n-1}$
     on $S^{2n+1}\times T^{m-n-1}$ to $D^{2n+2}\times T^{m-n-1}$.
     So $\mathcal{Z}_{P}$ is equivariantly cobordant to zero.
     The similar argument works for $\mathbb{R}\mathcal{Z}_{P}$.
\end{proof}

Next, we consider some ``lower level'' classes
$\mathcal{Z}_P/H_{\mathbb{C}}$ and $\mathbb{R}\mathcal{Z}_P/H_{\mathbb{R}}$.
For the diagonal action,
it is well-known that:
$D_{\Bbb C}=\{(g, g,\cdots, g)\}\subseteq T^m\curvearrowright \mathcal{Z}_P$ is free, and
$D_{\Bbb R}=\{(g, g,\cdots, g)\}\subseteq (\mathbb{Z}_2)^m\curvearrowright {\Bbb R}\mathcal{Z}_P$ is free. We have the following result parallel to Theorem~\ref{Thm:Surgery}.

\begin{thm}
Lickorish type construction for quotient spaces induced by the diagonal action on (real) moment-angle manifolds.
	\begin{itemize}
		\item ${\Bbb C}$-case:
			$\big\{$All quotient spaces $\mathcal{Z}_P/D_{\Bbb C}$ $\big\} /\sim_H\subseteq AS\big\{ {\Bbb C}P^n = \mathcal{Z}_{\Delta^n}
			\slash D_{\Bbb C} \,\big| \,\widetilde{\digamma}_{(a, b)}^{D_{\Bbb C}}$,
			$a+b=n+1,  1\leq a, b \leq n\big\}$;
	     \vskip .1cm
		\item ${\Bbb R}$-case:
			$\big\{$All quotient spaces ${\Bbb R}\mathcal{Z}_P/D_{\Bbb R}$ $\big\} /\sim_H\subseteq AS\big\{ {\Bbb R}P^n=\mathbb{R}\mathcal{Z}_{\Delta^n}
			\slash D_{\Bbb R}\,\big|\, $
			$\widetilde{\digamma}_{(a, b)}^{D_{\Bbb R}}, a+b=n+1, 1\leq a,b \leq n\big\}$
		\end{itemize}
\noindent where $P$ runs over all possible  $n$-dimensional simple polytopes.
\end{thm}
The operations $\widetilde{\digamma}_{(a, b)}^{D_{\Bbb C}}$ and
$\widetilde{\digamma}_{(a, b)}^{D_{\Bbb R}}$ are defined as follows.
$$\widetilde{\digamma}^{D_{\mathbb{C}}}_{(a, b)}(\mathcal{Z}_P)=
\begin{cases}
\big((\mathcal{Z}_P \slash D_{\mathbb{C}})\setminus (S^{2a-1}\times D^{2b}\times T^{m-n-2}) \big)& \\
\cup (D^{2a}\times S^{2b-1}\times T^{m-n-2}), & \text{ if } a\not=1,n;\\
\big((\mathcal{Z}_P\slash D_{\mathbb{C}})\times S^1 \big)\setminus(S^1\times D^{2n}\times T^{m-n-1})&\\
\cup(D^2\times S^{2n-1}\times T^{m-n-1}),& \text{ if } a=1;\\
 \text{converse operation of $a=1$ case},              & \text{ if } a=n.
\end{cases}$$
\[  \widetilde{\digamma}^{D_{\mathbb{R}}}_{(a, b)} (\mathbb{R}\mathcal{Z}_{P}\slash D_{\mathbb{R}}) =
 \begin{cases}
 \big((\mathbb{R}\mathcal{Z}_{P}\slash D_{\mathbb{R}}) \backslash
	     (S^{a-1}\times  D^{b} \times (\mathbb{Z}_2)^{m-n-2}) \big)  & \\
	     \cup
	       \big(D^a\times  S^{b-1} \times (\mathbb{Z}_2)^{m-n-2}) \big), & \text{ if } a\not=1, n;\\
 \big( (\mathbb{R}\mathcal{Z}_P\slash D_{\mathbb{R}} \times \mathbb{Z}_2 ) \backslash (S^0\times D^n\times (\mathbb{Z}_2)^{m-n-1})\big) &\\
  \cup \big(D^1\times S^{n-1}\times (\mathbb{Z}_2)^{m-n-1}\big),  & \text{ if } a=1;\\
 \text{converse operation of $a=1$ case},              & \text{ if } a=n.
\end{cases}
	   \]
\begin{rem}
In this section, we discuss the liftings from the combinatorial type surgeries to the equivariant surgeries on the highest level. 	These liftings are one-to-one correspondence between equivariant homeomorphsim classes
and combinatorial equivalent classes.
Hence we could say these surgeries are still constructible, since one can produce these ``highest level'' objects by concrete date from these concrete combinatorial operations.
\end{rem}

\section{Construction of quasitoric manifolds and small covers}

\subsection{Low dimensional cases I: quasitoric manifolds}
In the case of 2-dimensional simple polytopes, P. Orlik and F. Raymond's work \cite{OR} implies the following (also see~\cite[p.427]{DJ}).
\begin{thm} [Orlik-Raymond~\cite{OR}]\label{Orlik-Raymond}
			$$\{\text{All $4$-dim quasitoric manifolds}\big\}/\sim_D=AS\big\{ {\Bbb C}P^2,
			\overline{{\Bbb C}P}^2, S^2\times S^2,{\Bbb C}P^2\sharp
			\overline{{\Bbb C}P}^2\big|\,
		\widetilde{\sharp}\big\}$$
		where $\sim_D$ denotes the equivalence relation of $T^2$-diffeomorphism and
		$\widetilde{\sharp}$ is the equivariant
		connected sum of two manifolds.
\end{thm}

In the case of $3$-dimensional small covers, Izmestiev~\cite{Izmestiev}
 studied a special class of $3$-dimensional small covers whose characteristic functions
	take values in a basis of	$(\mathbb{Z}_2)^3$ (i.e. three linearly independent elements of $(\mathbb{Z}_2)^3$).		
 Izmestiev~\cite{Izmestiev} gave a Lickorish type construction of such $3$-dimensional small covers as follows.	
	\begin{thm} [Izmestiev~\cite{Izmestiev}]
		\		
		\begin{itemize}
			\item Combinatorial case:
			$$\mathcal{C}=\big\{(P^3, \lambda)\ | \ |\mathrm{Im} \lambda|=3\big\}=AS\big\{(I^3, \lambda_0) \ \text{with}\ |\mathrm{Im} \lambda_0|=3\ |\
			\sharp, \natural\big\}$$
			where $\sharp$ is the connected sum of two simple polytopes at some vertices and
			$\natural$ is the operation on a $3$-polytope
			shown in Figure~\ref{p:Dehn}.\vskip .1cm
			\item Topological case:
			$$\big\{ M(P^3, \lambda) \ | \ (P^3, \lambda)\in\mathcal{C}\big\}/=AS
			    \big\{M(I^3, \lambda_0)=T^3  \ | \
			\widetilde{\sharp}, \widetilde{\natural}\, \big\}$$
			where $\widetilde{\sharp}$ and  $\widetilde{\natural}$ are the equivariant
			connected sum and the equivariant $\frac{0}{1}$-type Dehn surgery, respectively.
		\end{itemize}
	\end{thm}
	
	 \begin{figure}
         \includegraphics[width=0.45\textwidth]{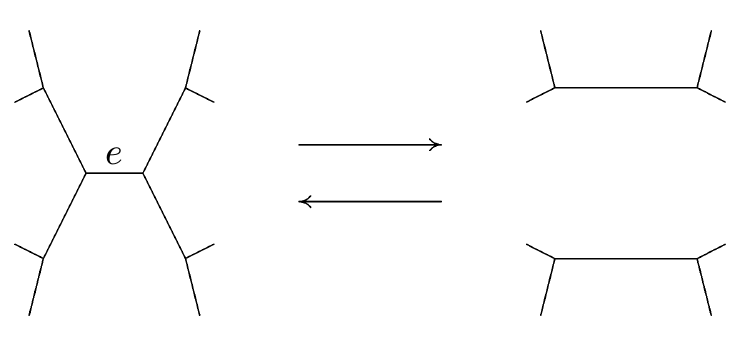}\\
          \caption{}\label{p:Dehn}
      \end{figure}

In dimension 6, Shintar\^{o} Kuroki discussed the equivariant diffeomorphism classification question of all 6-dimensional torus manifolds with vanishing odd-degree cohomology, which is a wider class of manifolds including quasitoric manifolds. Similar to Theorem~\ref{Orlik-Raymond},
\begin{thm} [Kuroki~\cite{Kuroki16}]
Up to $T^3$-diffeomorphism, we have one-to-one correspondence:

$ \big\{  \text{All $6$-dimensional 1-connected torus manifolds with vanishing odd-degree}$

$ \text{\ cohomology} \big\}/\sim_D$ $=AS\big\{
S^6, S^4 \text{-bundles over~}S^2, \text{quasitoric 6-manifolds~}\,\big|\,~
	\widetilde{\sharp}~\big\}$,
	where $\widetilde{\sharp}$ is the equivariant
	connected sum of two manifolds and $\sim_D$ denotes the equivalence relation of equivariant $T^k$-diffeomorphism.
\end{thm}

\subsection{Low dimensional case II: small covers}

The Four Color Theorem tells us that any simple
$3$-polytope admits $(\mathbb{Z}_2)^3$-colorings.
Denote by
 \begin{itemize}
  \item $\mathcal{P}:=$ the set of all pairs $(P^3, \lambda)$, where
	$P^3$ is a $3$-dimensional simple convex polytope and $\lambda$ is a
	nondegenerate $(\mathbb{Z}_2)^3$-coloring on it.\vskip .1cm
	
 \item $\mathcal{M}:=$ the set of all $3$-dimensional small covers.
\end{itemize}
By Davis-Januszkiewicz~\cite{DJ}, there exists a one-to-one correspondence:
\begin{align*}
		\mathcal{P} & \longleftrightarrow\mathcal{M} \\
		 (P^3, \lambda) &\longmapsto M(P^3,\lambda)
\end{align*}
	
	Zhi L\"u and Li Yu studied general $3$-dimensional small covers in~\cite{LY}
	 and showed
	that any $3$-dimensional small cover can be obtained from
	 small covers over $\Delta^3$ and a triangular prism
	via a sequence of surgeries.
	Combinatorially, one has:
\begin{thm} [L\"u-Yu~\cite{LY}]
		All pairs $(P^3, \lambda)$ of $\mathcal{P}$ form  an algebraic system with generators
		$(\Delta^3,\sigma\circ\lambda_0)$, $(P^3(3),\sigma\circ\lambda_1)$,
		$(P^3(3),\sigma\circ\lambda_2)$, $(P^3(3),\sigma\circ\lambda_3)$,
		$(P^3(3),\sigma\circ\lambda_4)$, $\sigma\in \text{\rm
			GL}(3,\mathbb{Z}_2)$ and
			 six operations $\sharp^v$, $\sharp^e$,
		$\sharp^{eve}$,$\natural$, $\sharp^\triangle$, $\sharp^\copyright$,
		where $\Delta^3$ is a $3$-simplex and $P^3(3)$ is a triangular prism,
		and $\lambda_0,\cdots,\lambda_4$ are shown in Figure~\ref{p:basic-pair} .
\end{thm}

   \begin{figure}
         \includegraphics[width=0.86\textwidth]{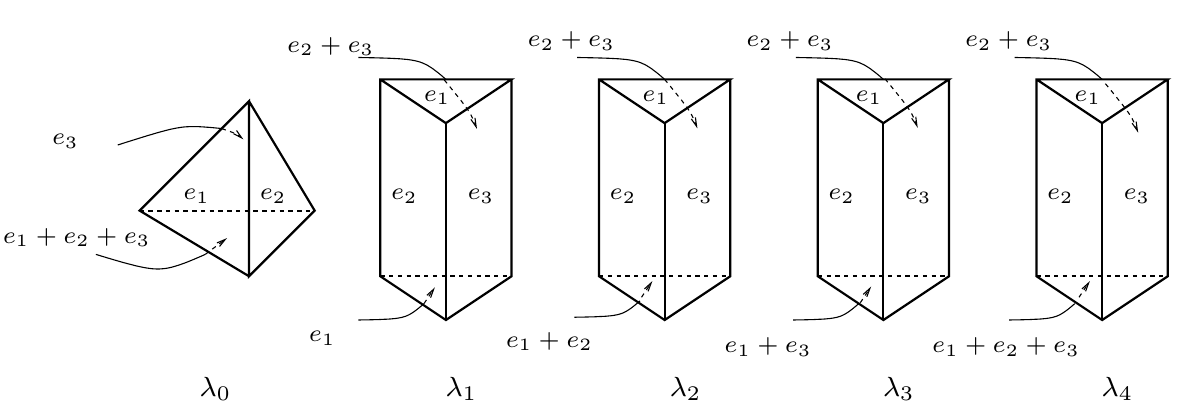}\\
          \caption{The basic pairs where $\{e_1,e_2,e_3\}$
           is the standard basis of $(\mathbb{Z}_2)^3$}\label{p:basic-pair}
      \end{figure}

The six operations $\sharp^v,
\sharp^e$, $\sharp^{eve}$, $\natural$, $\sharp^\triangle$,
$\sharp^\copyright$ on $\mathcal{P}$ are shown below (cf. \cite{LY}).
\begin{itemize}
 \item Operation $\sharp^v$ on $\mathcal{P}$
\[
\includegraphics[scale=0.9]{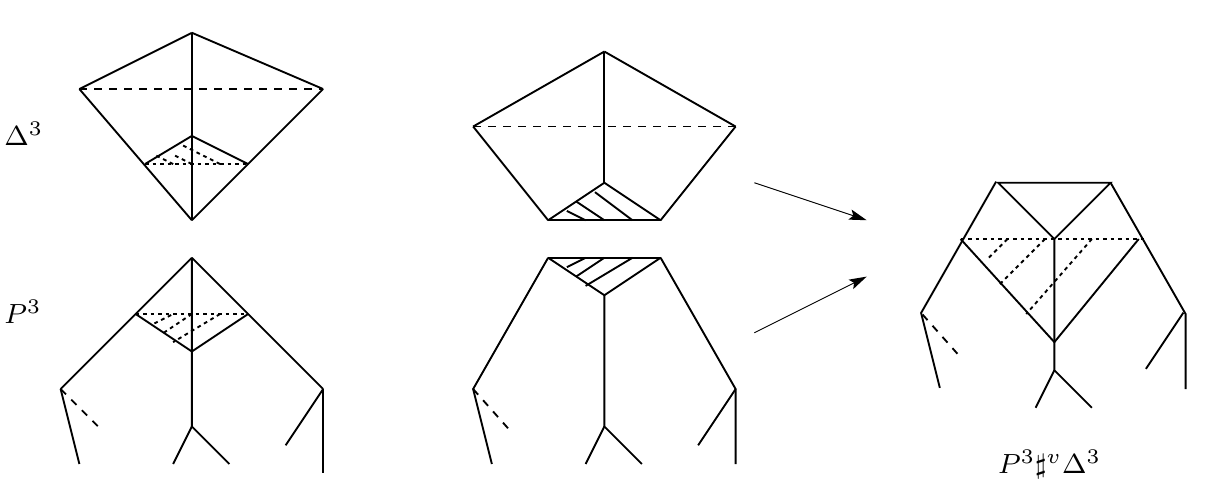}
\]
 \item Operation $\sharp^e$ on $\mathcal{P}:$
\[
\includegraphics[scale=0.9]{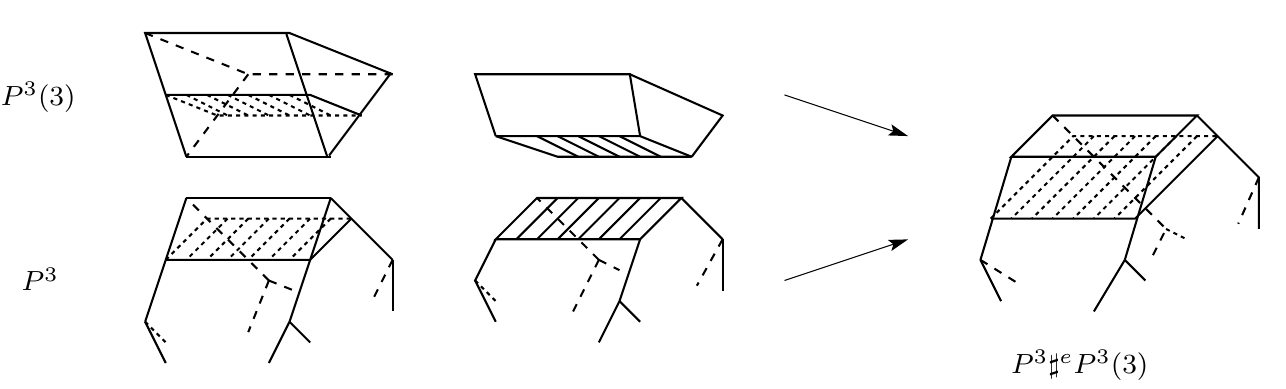}
\]
 \item Operation $\sharp^{eve}$ on $\mathcal{P}:$
\[
\includegraphics[scale=1.06]{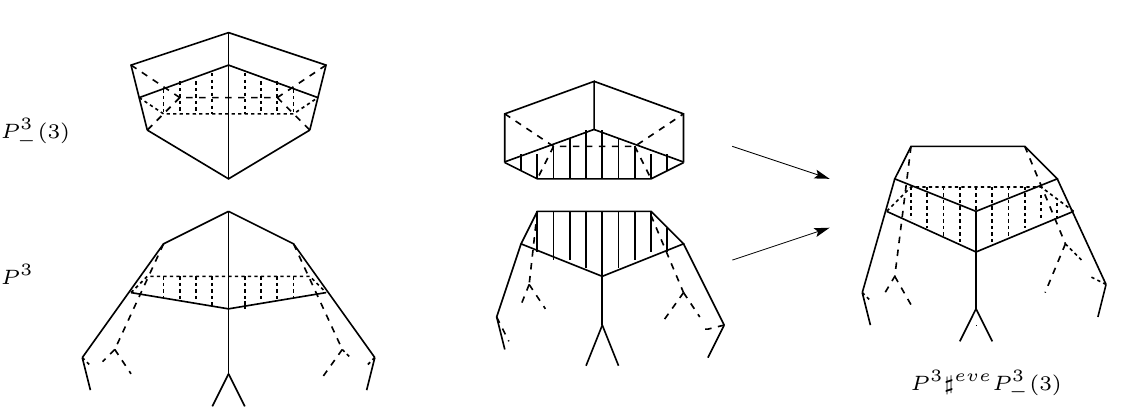}
\]
Here $P^3_{-}(3)$ is obtained by cutting a vertex from the triangular prism $P^3(3)$.
 \item Operation $\natural$ on $\mathcal{P}:$
\[
\includegraphics[scale=1]{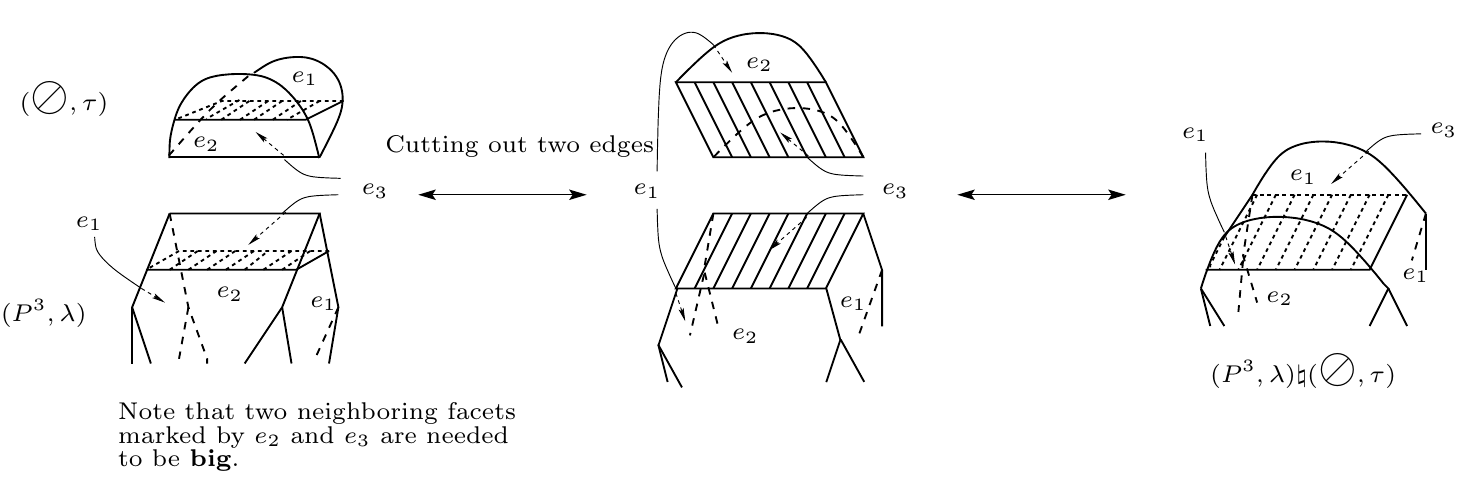}
\]
\item Operation $\sharp^\triangle$ on $\mathcal{P}:$
\[
\includegraphics[scale=1.15]{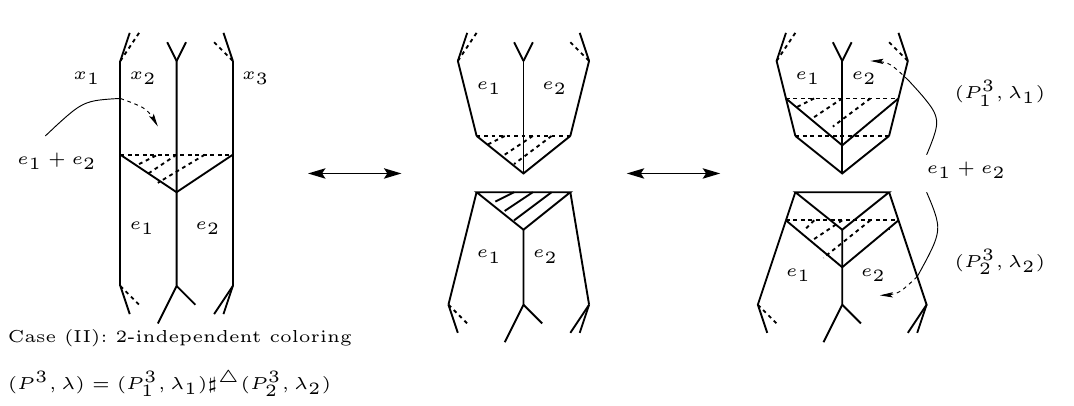}
\]
\item Operation $\sharp^\copyright$ on $\mathcal{P}:$
\[
\includegraphics[scale=1.0]{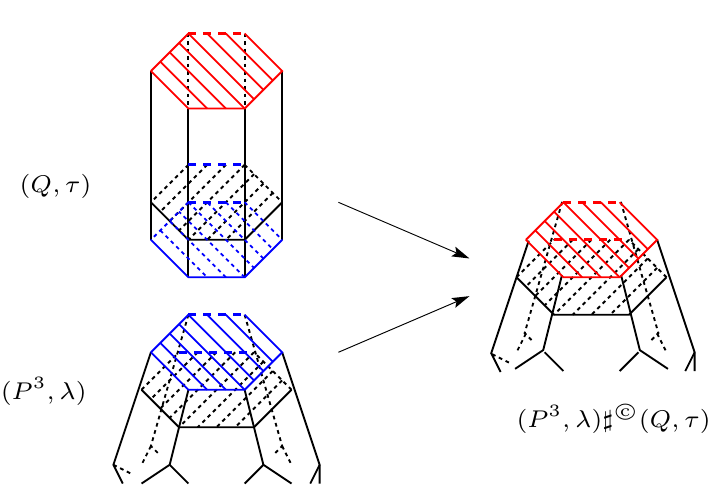}
\]
\end{itemize}

By the construction of small covers from pairs in $\mathcal{P}$,
  we have the following operations on $\mathcal{M}$ corresponding to
  $\sharp^v, \sharp^e,
	\sharp^{eve}, \natural, \sharp^\triangle, \sharp^\copyright$.
\[\sharp^v, \sharp^e,
	\sharp^{eve}, \natural, \sharp^\triangle, \sharp^\copyright \text{
		on } \mathcal{P}\longleftrightarrow \widetilde{\sharp^v},
	\widetilde{\sharp^e}, \widetilde{\sharp^{eve}},
	\widetilde{\natural}, \widetilde{\sharp^\triangle},
	\widetilde{\sharp^\copyright} \text{ on } \mathcal{M}
\]
\begin{itemize}
 \item	$\widetilde{\sharp^v}$ is the equivariant connected sum.\vskip .1cm

 \item $\widetilde{\natural}$ is
  equivariant $\frac{0}{1}$-type Dehn surgery.\vskip .1cm

 \item $\widetilde{\sharp^e}, \widetilde{\sharp^{eve}},
		\widetilde{\sharp^\triangle}, \widetilde{\sharp^\copyright}$ are some
		equivariant cut-and-paste operations which can be
		understood as the generalized equivariant connected sums.
\end{itemize}		

For the generators of $\mathcal{M}$, we take
		$M(\Delta^3,\sigma\circ\lambda_0)$ and
		$M(P^3(3),\sigma\circ\lambda_i) (i=1,...,4), \sigma\in \text{\rm
			GL}(3,\mathbb{Z}_2)$,  which give all elementary generators of
		the algebraic system $\langle\mathcal{M}; \widetilde{\sharp^v}$,
		$\widetilde{\sharp^e}$, $\widetilde{\sharp^{eve}}$,
		$\widetilde{\natural}$, $\widetilde{\sharp^\triangle}$,
		$\widetilde{\sharp^\copyright}\rangle$.
The topological types of these generators are as follows:
\begin{itemize}
\item $M(\Delta^3,\lambda_0)\approx\mathbb{R}P^3 \text{ with the canonical linear
			$(\mathbb{Z}_2)^3$-action}$
		\item
		$M(P^3(3),\lambda_i)(i=1,...,4)\approx S^1\times\mathbb{R}P^2 \text{ with four
			different
			$(\mathbb{Z}_2)^3$-actions}$
		\end{itemize}
So we have the Lickorish type construction of all $3$-dimensional small covers.
\begin{thm}[L\"u-Yu~\cite{LY}] \label{Thm:Lu-Yu}
  All the $3$-dimensional small covers form an algebraic system with generators
		$\mathbb{R}P^3$ and $S^1\times \mathbb{R}P^2$ with certain
		$(\mathbb{Z}_2)^3$-actions and six operations
		$\widetilde{\sharp^v}, \widetilde{\sharp^e},
		\widetilde{\sharp^{eve}}, \widetilde{\natural},
		\widetilde{\sharp^\triangle}, \widetilde{\sharp^\copyright}$.
\end{thm}
	 In addition, Kuroki \cite{Kuroki}
	studied the relations among  six operations $\sharp^v,
		\sharp^e$, $\sharp^{eve}$, $\natural$, $\sharp^\triangle$,
		$\sharp^\copyright$ on $\mathcal{P}$ and found that
		$\sharp^e = \natural \circ(\sharp^vP^3(3)) $ and $\sharp^{eve} =
		\natural^2 \circ (\sharp^vP^3_{-}(3)) $.
 Furthermore, Nishimura~\cite{Nishimura} discovered more
 relations among the operations in Theorem~\ref{Thm:Lu-Yu} and
 obtained another algebraic system by the following theroem,
 which improved Theorem~\ref{Thm:Lu-Yu}.
		\begin{thm}[Y. Nishimura~\cite{Nishimura}] \label{Thm:Nishimura}
			All the $3$-dimensional small covers form an algebraic system with  generators
			$T^3, \mathbb{R}P^3$ and $S^1\times \mathbb{R}P^2$ with certain
			$(\mathbb{Z}_2)^3$-actions and two operations
			$\widetilde{\sharp^v},  \widetilde{\natural}$.
		\end{thm}
		So by Theorem~\ref{Thm:Nishimura}, we can obtain any $3$-dimensional small cover
		 by connected sum and some special kind of
		 Dehn surgeries. 
\vskip .2cm
It should be pointed out that Nishimura in~\cite{N}  also gave Lickorish type construction for all orientable 3-dimensional small covers as follows: all orientable 3-dimensional small
covers are obtained from $\mathbb{R}P^3$ and $T^3$ by using the equivariant connected sum and  equivariant ${0\over 1}$-type Dehn surgery.

\begin{rem}
In this section, we see in the "lowest level", there are plenty results similar to Theorem \ref{LickorishThm}. Many
``Equivariant Dehn surgeries" admit more concrete correspond to their combinatorial surgeries.

\vskip .2cm

Furthermore,  in many cases, one could reduce the number of generators and the number of specific types of
equivariant surgeries as less as possible.
In particular, according to \cite{LY}, we could see these six operations (equivariant-surgeries) are very concrete and easy to handle in the
combinatorial level as well as the manifold level (lowest level).
\end{rem}

\section{Lickorish type problems in equivariant cobordisms}
Recall that two smooth closed $n$-manifolds $M_1$ and $M_2$ are bordant if their disjoint union are the boundary of some $n+1$ manifold. One knows that if $M_1$ is bordant to $M_2$, $M_2$ can be obtained from
$M_1$ by a finite steps of surgeries, which is also a kind of Lickorish type construction. Therefore, in cobordism,
the idea of Lickorish type construction may provide a good point of view to discuss the manifolds in the same bordism classes.
\vskip .2cm

In equivariant case, questions become more difficult. Roughly speaking, let $G$ be a compact Lie group, two $G$-manifolds $M_1$ and $M_2$ are $G$-equivariant bordant if there exists a $G$-manifolds $W$ with boundary $M_1\sqcup M_2$ such that their $G$-structures are equivalent. We are interested in the equivariant cobordism classification problems of $G$-manifolds.
\vskip .2cm
We can apply the preceding discussion of the Lickorish type construction of small covers
to study their equivariant cobordism classification. Define
 $\widehat{\mathcal{M}}$ to be the set consisting of  equivariant unoriented
cobordism classes of all $3$-dimensional small covers. 	
Since $\widehat{\mathcal{M}}$ forms an abelian group under disjoint union,
 we can think of $\widehat{\mathcal{M}}$ as a vector space over $\mathbb{Z}_2$.
\vskip .2cm

Let $[M(P_1^3, \lambda_1)]$ and $[M(P_2^3, \lambda_2)]$ be two
classes in $\widehat{\mathcal{M}}$ where $\lambda_i$ is a
characteristic function on a $3$-dimensional simple polytope $P^3_i$, $i=1,2$. From L\"u-Yu~\cite{LY}, we know
\begin{align*} &[M(P_1^3,
\lambda_1)\widetilde{\sharp^v} M(P_2^3, \lambda_2)]=[M(P_1^3,
\lambda_1)]+[M(P_2^3, \lambda_2)]\\
&[M(P^3,\lambda)\widetilde{\sharp^e}M(P^3(3),
\tau)]=[M(P^3,\lambda)]+[M(P^3(3), \tau)]\\
&[M(P^3,\lambda)\widetilde{\sharp^{eve}}M(P^3_-(3),
\tau)]=[M(P^3,\lambda)]+[M(P^3_-(3), \tau)]\\
\end{align*}
\begin{align*}
&[M(P^3,\lambda)\widetilde{\natural}M(\text{\large $\oslash$},
\tau)]=[M(P^3,\lambda)]\\
&[M(P^3,\lambda)\widetilde{\sharp^\copyright}M(P^3(i),
\tau)]=[M(P^3,\lambda)]+[M(P^3(i), \tau)], i=3,4,5.\\
&[M(P_1^3, \lambda_1)]\widetilde{\sharp^\triangle}[M(P_2^3,
\lambda_2)]\\=&\begin{cases} [M(P_1^3, \lambda_1)]+[M(P_2^3,
\lambda_2)] \\
\text{or } [M(P_1^3, \lambda_1)]+[M(P_2^3,
\lambda_2)]+[M(P^3(3),\lambda_1\sharp^\triangle\lambda_2)].
\end{cases}
\end{align*}
By the above discussion, it is easy to see that
 the abelian group $\widehat{\mathcal{M}}$ is generated by
 some small covers over $\Delta^3$ and $P^3(3)$.
			
\begin{prop}[Equivariant cobordism classification of $3$-dim small covers]
The abelian group
		$\widehat{\mathcal{M}}$ is generated by classes of $\mathbb{R}P^3$
		and $S^1\times \mathbb{R}P^2$ with certain
		$(\mathbb{Z}_2)^3$-actions.
\end{prop}

It is known  in~\cite{Lu} that as a $\mathbb{Z}_2$-vector space, $\widehat{\mathcal{M}}$ has dimension 13.
\vskip .2cm

Similarly, for quasitoric manifolds we would like to study the following problem.
\begin{prob}
	Give a Lickorish type construction for all $6$-dimensional quasitoric manifolds with simple generators
	and simple operators,
	and compute the equivariant cobordism group.
\end{prob}

\begin{prob}
	Find some simple Lickorish type constructions between two quasitoric manifolds  if they are equivariantly bordant.
\end{prob}

\vskip .2cm

In this survey, we've introduced and restated many theorems in the ``Lickorish type'' style. In our point of view, Lickorish's original idea is to use less generators and finite ``easy'' operators to construct all the objects. We are happy to find many ``Lickorish type'' combinatorial and geometric objects and we hope that these ``Lickorish type'' style theorems could provide another point of view to understand combination, geometry and topology.

\end{document}